\documentclass[reqno]{amsart}
\usepackage[english]{babel}
\usepackage{amssymb,verbatim,xcolor,cite}
\usepackage[T1]{fontenc}
\newtheorem{theorem}{Theorem}[section]

\newtheorem{lemma}[theorem]{Lemma}
\newtheorem{corollary}[theorem]{Corollary}
\theoremstyle{definition}
\newtheorem{definition}[theorem]{Definition}
\newtheorem{example}[theorem]{Example}
\theoremstyle{remark}
\newtheorem*{remark}{Remark}

\newtheorem*{ackn}{Acknowledgements}

\numberwithin{equation}{section}

\newcommand{\UU}{\operatorname{U}}
\newcommand{\HH}{\operatorname{H}_{m}}

\begin{document}

\title[Complex Hessian Equations and the Dinew-Ko{\l}odziej Estimate]{Continuity of Solutions to Complex Hessian Equations via the Dinew-Ko{\l}odziej Estimate}

\author{Per \AA hag}\address{Department of Mathematics and Mathematical Statistics\\ Ume\aa \ University\\SE-901~87 Ume\aa \\ Sweden}\email{per.ahag@umu.se}

\author{Rafa\l\ Czy{\.z}}\address{Faculty of Mathematics and Computer Science, Jagiellonian University, \L ojasiewicza~6, 30-348 Krak\'ow, Poland}
\email{rafal.czyz@im.uj.edu.pl}

\keywords{Complex Hessian equation, Dinew-Ko{\l}odziej estimate, $m$-subharmonic function, regularity}
\subjclass[2020]{Primary 31C45, 35B65, 35B35; Secondary  32U05, 35J60.}

\begin{abstract}
This study extends the celebrated volume-capacity estimates of Dinew and Ko{\l}odziej, providing a foundation for examining the regularity of solutions to boundary value problems for complex Hessian equations. By integrating the techniques established by Dinew and Ko{\l}odziej and incorporating recent advances by Charabati and Zeriahi, we demonstrate the continuity of the solutions.
\end{abstract}

\maketitle

\centerline {\bf \today}

\section{Introduction}
Let $K\subset\mathbb{C}$ be a compact subset
in the complex plane with area $A(K)$ and logarithmic capacity $c(K)$. In 1928, Pólya~\cite{Polya28} established the following inequality:
\begin{equation*}
A(K) \leq \pi c(K)^2.
\end{equation*}
which has found widespread use and has been generalized in various contexts within analysis and geometry. For example, it was extended in~\cite{ACKPZ09} and used to affirm a conjecture proposed by Demailly~\cite{Dem09}.

Of particular interest to this paper is the work by Dinew and Ko\l odziej~\cite{DK14}, who proved that the volume of a relatively compact set in $\mathbb{C}^n$ can be estimated using the Hessian capacity:
\begin{equation}\label{intr:DK}
V_{2n}(K) \leq C(\operatorname{cap}_m(K))^{\alpha},
\end{equation}
where $dV_{2n}$ denotes the Lebesgue measure in $\mathbb{R}^{2n}$,  $1 < \alpha < \frac{n}{n-m}$ and $\operatorname{cap}_m(K)$ is defined as the Hessian capacity. This result is crucial for studying the complex Hessian operator on compact K\"ahler manifolds.

In Theorem~\ref{DK}, we extend their volume-capacity estimate using techniques from Orlicz theory.

\medskip

\noindent{\textbf{Theorem~\ref{DK}.}}
\emph{Let $\Omega$ be a bounded $m$-hyperconvex domain in $\mathbb C^n$. Then for any $0<\epsilon\leq \frac {n+1}{3n}$ there exist constants $C_1, C_2>0$ such that for all $K\Subset \Omega$ it  holds:
\begin{equation}\label{int:est}
V_{2n}(K)\leq C_1\operatorname {cap}_m(K)^{\frac {n}{n-m}}\operatorname{W}_0\left(C_2\operatorname {cap}_m(K)^{\frac {-1}{m(1+\epsilon)}}\right)^{\frac {nm(1+\epsilon)}{n-m}},
\end{equation}
where $\operatorname{W}_0$ is the Lambert $W$ function.}

\medskip

In 1986, Vinacua \cite{Vin86} (see also \cite{Vin88}) expanded upon the foundational work of Caffarelli, Nirenberg, and Spruck \cite{CNS85} by introducing complex Hessian equations:
\begin{equation}\label{Intr:HH}
\HH(u) = \mu,
\end{equation}
bridging classical and modern theories in partial differential equations. As defined (Definition~\ref{m-sh}), the $1$-Hessian operator, $\operatorname{H}_1$, corresponds to the classical Laplace operator on subharmonic functions, herein referred to as $1$-subharmonic functions. Similarly, the $n$-Hessian operator, $\operatorname{H}_n$, corresponds  to the complex Monge-Ampère operator for plurisubharmonic functions. For $k = 2, \dots, n-1$, the $m$-Hessian operators form a sequence of partial differential operators, spanning from the Laplace to the complex Monge-Ampère operators.

The integrability of general $m$-subharmonic functions significantly differs from that of $n$-subharmonic functions. While all plurisubharmonic functions are locally $L^p$ integrable for any $p > 0$, this is not necessarily true for $m$-subharmonic functions. B\l ocki \cite{Blo05} conjectured that $m$-subharmonic functions should be locally $L^p$ integrable for $p < \frac{nm}{n-m}$, a conjecture partially verified in \cite{AC20,DK14}.

The complex Hessian equation and $m$-subharmonic functions have attracted widespread attention. A significant advancement was made by B\l ocki in 2005, who extended these concepts to non-smooth admissible functions and introduced pluripotential methods \cite{Blo05}. More recently, Lu adapted Cegrell's framework~\cite{Ceg98,Ceg04,Ceg08} for the complex Hessian equations~\cite{Lu12,Lu13S,Lu13V,Lu15}.

In this paper, we build upon the techniques of Dinew and Kołodziej~\cite{DK14} and incorporate recent insights from Charabati and Zeriahi~\cite{CZ23} to address the regularity of solutions  for the complex Hessian equations.

 Consider a bounded strictly $m$-pseudoconvex domain $\Omega \subset \mathbb{C}^n$. For a given density function $f$ and a boundary value function $g \in \mathcal{C}(\partial \Omega)$, the problem of interest is:
\begin{equation}
\label{Intr: DP2}
\begin{aligned}
\HH(U(f,g)) &= f dV_{2n}, \\
\lim_{z \to w} U(f,g)(z) &= g(w), \quad \text{for all } w \in \partial \Omega.
\end{aligned}
\end{equation}
If the density function $f \in L^{p}$ for $p > \frac{n}{m}$, then $\UU(f,g)$ is continuous and has been subsequently proven to be Hölder continuous, as documented in~\cite{DK14, Bou19, Cha16, KN20, N14}.
If $p<\frac{n}{m}$, then the solution to~(\ref{Intr: DP2}) need not to be bounded~\cite{DK14}, which is a significant contrast to the case when $m=n$. In~\cite{AC22}, Cegrell's energy classes played a crucial role in studying the regularity of unbounded $m$-subharmonic functions.

Our goal is to investigate further the regularity of solutions to~\eqref{Intr: DP2} when the density function $f$ belongs to the Orlicz space $L^{n/m}(\log L)^\alpha$ with $\alpha > 2n$. We establish the following result:

\medskip

\noindent\textbf{Theorem~\ref{t1}.} \emph{Let $\Omega \subset \mathbb{C}^n$ be a bounded strictly $m$-pseudoconvex domain, let $f \in L^{\frac{n}{m}}(\log L)^{\alpha}$ for $\alpha > 2n$, and let $g \in \mathcal{C}(\partial \Omega)$. Then, the unique solution $U(f, g)$ of the Dirichlet problem for the complex Hessian operator~\eqref{Sec5: DP} is continuous on $\bar{\Omega}$. Moreover, the following estimate holds:
\begin{multline*}
\|U(f_1,g_1)-U(f_2,g_2)\|_{\infty} \leq \|g_1-g_2\|_{\infty}+ C_1 \|f_1-f_2\|_{\alpha}^{-\frac{1}{\gamma}}\\
+ C_2 e_{m,m}(U(|f_1-f_2|,0))^{\frac{1}{2m}} \exp\left(C_3 \|f_1-f_2\|_{\alpha}^{-\frac{1}{\gamma}}\right),
\end{multline*}
where $\gamma = (1+\epsilon)m - \frac{\alpha m}{n}$ with $0 < \epsilon < \min\left(\frac{n+1}{3n},\frac {\alpha}{n}-2\right)$. Here, $\|f\|_{\alpha}$ denotes the norm in $L^{\frac{n}{m}}(\log L)^{\alpha}$, and $C_1, C_2, C_3$ are positive universal constants. Moreover, $e_{m,m}(u)$ is defined as $\int_{\Omega} (-u)^m \HH(u)$.}

\medskip

This paper is organized as follows: Section~\ref{sec:background} provides the necessary background on complex Hessian equations and the theoretical framework used throughout the paper, including discussions on the theory of Orlicz spaces. Section~\ref{Sec3: DK} contains the proof of Theorem~\ref{DK}, while in Section~\ref{Sec: Continuous Solution}, we present the proof of Theorem~\ref{t1}. The final section, Section~\ref{section_bs}, studies the cases when $n<\alpha \leq 2n$ and $\alpha\leq n$.

\medskip

We want to emphasize that the proof of Theorem~\ref{DK} and the results in Section~\ref{section_bs} rely heavily on the properties of plurisubharmonic functions. Notably, significant insights into $m$-subharmonic functions can be gained by examining the subset of plurisubharmonic functions, a phenomenon first observed by Dinew and Kołodziej~\cite{DK14}.

\medskip

\begin{ackn}
We extend our heartfelt appreciation to Chinh H. Lu for his invaluable insights and discussions, which greatly enriched a   preliminary version of this paper.
\end{ackn}

\section{Preliminaries}\label{sec:background}

This section is organized as follows: Section~\ref{SS: Cegrell Classes} provides the fundamental definitions of the generalized potential theory we are interested in. Section~\ref{SS: Orlicz Spaces} offers basic information on Orlicz spaces, specifically $L^{\frac{n}{m}}(\log L)^{\alpha}$. Finally, Section~\ref{SS: Lambert} reviews some essential aspects of the Lambert $W$ function.

\subsection{Generalized potential theory}\label{SS: Cegrell Classes}

This subsection provides foundational definitions and results for $m$-subharmonic functions and the complex Hessian operator. Consider $\Omega \subset \mathbb{C}^n$, where $n \geq 2$, as a bounded domain, and let $1 \leq m \leq n$. Define $\mathbb{C}_{(1,1)}$ as the set of $(1,1)$-forms with constant coefficients. We then define
\[
\Gamma_m = \left\{\alpha \in \mathbb{C}_{(1,1)} : \alpha \wedge (dd^c|z|^2)^{n-1} \geq 0, \ldots, \alpha^m \wedge (dd^c|z|^2)^{n-m} \geq 0 \right\}.
\]

\begin{definition}\label{m-sh}
Let $n \geq 2$, and $1 \leq m \leq n$. Assume $\Omega \subset \mathbb{C}^n$ is a bounded domain. A function $u$, defined on $\Omega$ and subharmonic, is said to be \emph{$m$-subharmonic} if it satisfies
\[
dd^c u \wedge \alpha_1 \wedge \dots \wedge \alpha_{m-1} \wedge (dd^c|z|^2)^{n-m} \geq 0,
\]
in the sense of currents for all $\alpha_1, \ldots, \alpha_{m-1} \in \Gamma_m$. We denote the set of all such $m$-subharmonic functions on $\Omega$ by $\mathcal{SH}_m(\Omega)$.
\end{definition}

\begin{definition}\label{prel_hcx}
Let $n \geq 2$, and $1 \leq m \leq n$. A bounded domain $\Omega \subset \mathbb{C}^n$ is termed \emph{$m$-hyperconvex} if there is a non-negative, $m$-subharmonic exhaustion function, i.e., there exists an $m$-subharmonic function $\varphi: \Omega \to (-\infty,0]$ such that for every $c < 0$, the closure of $\{z \in \Omega : \varphi(z) < c\}$ is compact within $\Omega$.
\end{definition}

For additional insights into $m$-hyperconvex domains, we refer to~\cite{ACH18}.

\begin{definition}\label{stricm-hyp}
An open set $\Omega \subset \mathbb{C}^n$ is \emph{strictly $m$-pseudoconvex} if it admits a smooth defining function $\rho$ which is strictly $m$-subharmonic in a neighborhood of $\bar{\Omega}$ and satisfies $|\nabla \rho| > 0$ at each point in $\partial\Omega = \{\rho = 0\}$.
\end{definition}

Next, we introduce function classes essential to this paper.
A function $\varphi$, defined on an $m$-hyperconvex domain $\Omega$ and $m$-subharmonic, belongs to $\mathcal{E}^0_{m}(\Omega)$ if it is bounded,
\[
\lim_{z \rightarrow \xi} \varphi(z) = 0 \text{ for every } \xi \in \partial \Omega,
\]
and
\[
\int_{\Omega} \operatorname{H}_m(\varphi) < \infty.
\]

\begin{definition}\label{CC}
Let $n \geq 2$, $1 \leq m \leq n$, and $p \geq 0$. A function $u$ defined on a bounded $m$-hyperconvex domain $\Omega$ in $\mathbb{C}^n$ belongs to $\mathcal{F}_{m}(\Omega)$ if there exists a decreasing sequence $\{\varphi_j\}$, $\varphi_j \in \mathcal{E}^0_{m}(\Omega)$, converging pointwise to $u$ as $j \to \infty$, and $\sup_j \int_{\Omega} \operatorname{H}_m(\varphi_j)< \infty$.
\end{definition}

In~\cite{Lu12, Lu15}, it was shown that for a function $u \in \mathcal{F}_{m}(\Omega)$, the complex Hessian operator $\operatorname{H}_m(u)$ is well-defined and given by
\[
\operatorname{H}_m(u) = (dd^c u)^m \wedge (dd^c|z|^2)^{n-m},
\]
where $d = \partial + \bar{\partial}$ and $d^c = \sqrt{-1} (\bar{\partial} - \partial)$.

Let us recall the Hessian capacity:
\[
\operatorname{cap}_{m}(E) = \sup \left\{ \int_E \operatorname{H}_m(u) : u \in \mathcal{SH}_m(\Omega), -1 \leq u \leq 0 \right\}.
\]

\subsection{Orlicz Spaces} \label{SS: Orlicz Spaces}

This subsection introduces some notations and elementary facts concerning Orlicz spaces, which will be useful in later discussions. This section is based on~\cite{M89}.

Let $\varphi: [0,\infty) \to [0,\infty)$ be an increasing, continuous, and convex function such that $\varphi(0) = 0$, $\lim_{t \to 0^+} \frac{\varphi(t)}{t} = 0$, and $\lim_{t \to \infty} \frac{\varphi(t)}{t} = \infty$. We shall call such a function \emph{admissible}.

Let $X$ be the space of measurable functions (with respect to the Lebesgue measure) on $\Omega$. Define a modular on $X$ by
\[
\rho(f) = \int_{\Omega} \varphi(|f|) dV_{2n},
\]
and introduce the Orlicz class:
\[
L_0^{\varphi} = \{f \in X : \rho(f) < \infty\}.
\]
The Orlicz space $L^{\varphi}$ is the smallest vector space containing $L_0^{\varphi}$. Moreover, $L^{\varphi}$ is a Banach space equipped with the Orlicz norm
\[
\|f\|^0_{\varphi} = \sup\left\{\int_{\Omega} |fg| dV_{2n} : \int_{\Omega} \varphi^*(|g|) dV_{2n} \leq 1\right\}
\]
or the equivalent Luxemburg norm
\[
\|f\|_{\varphi} = \inf\left\{\lambda > 0 : \int_{\Omega} \varphi\left(\frac{|f|}{\lambda}\right) dV_{2n} \leq 1\right\},
\]
where $\varphi^*$, the Legendre transform of $\varphi$ is defined as
\[
\varphi^*(s) = \sup_{t \geq 0} (st - \varphi(t)).
\]
The Legendre transform $\varphi^*$ is also called the complementary function in the sense of Young.

Let us now recall the following version of Young's inequality: for all $t, s \geq 0$,
\[
st \leq \varphi(t) + \varphi^*(s),
\]
and note that
\[
\|f\|_{\varphi} \leq \|f\|^0_{\varphi} \leq 2\|f\|_{\varphi}.
\]
Hence, by the Young inequality and the definition of the Orlicz norm,
\begin{multline}\label{o2}
\|f\|^0_{\varphi} = \sup_{g} \left(\int_{\Omega} |fg| dV_{2n} : \int_{\Omega} \varphi^*(|g|) dV_{2n} \leq 1\right) \\ \leq \sup_{g} \left(\int_{\Omega} \varphi(|f|) + \varphi^*(|g|) dV_{2n}: \int_{\Omega} \varphi^*(|g|) dV_{2n} \leq 1 \right) \\ \leq \int_{\Omega} \varphi(|f|) dV_{2n} + 1.
\end{multline}

We present the following counterpart to the classical H\"older's inequality:

\begin{theorem}\label{hol}
If $f \in L^{\varphi}$ and $g \in L^{\varphi^*}$, then
\[
\left|\int_{\Omega} fg dV_{2n}\right| \leq \|f\|^0_{\varphi} \|g\|_{\varphi^*}, \quad \text{and} \quad \left|\int_{\Omega} fg dV_{2n}\right| \leq \|f\|_{\varphi} \|g\|^0_{\varphi^*}.
\]
\end{theorem}

Let $1 \leq m \leq n$ be integers, $n\geq 2$, and let $\alpha > 0$. A central tool in this paper is the Orlicz space generated by the function
\[
\varphi(t) = (1+t)^{\frac{n}{m}} (\log(1+t))^{\alpha},
\]
which will be denoted by $L^{\frac{n}{m}}(\log L)^{\alpha}$. The corresponding Orlicz norm will be denoted for simplicity by $\|\cdot\|_{\alpha}$.

The following example will be used in Theorem~\ref{DK} and Theorem~\ref{t1}

\begin{example}
Let $K \subset \Omega$ be such that $0 < V_{2n}(K) < \infty$, and let $\varphi$ be an admissible function. Then,
\[
\|\chi_K\|_{\varphi} = \frac{1}{\varphi(V_{2n}(K)^{-1})},
\]
and
\[
\|\chi_K\|_{\varphi}^0 = V_{2n}(K)(\varphi^*)^{-1}\left(V_{2n}(K)^{-1}\right),
\]
where $(\varphi^*)^{-1}$ denotes the inverse function to $\varphi^*$. If $f \in L^{\varphi}$, then we have:
\begin{equation}\label{o1}
\int_K f dV_{2n} \leq \|f\|_{\varphi} \|\chi_K\|_{\varphi^*}^0 = \|f\|_{\varphi} V_{2n}(K) \varphi^{-1}(V_{2n}(K)^{-1}),
\end{equation}
where $\varphi^{-1}$ is the inverse function to $\varphi$. $\hfill{\Box}$
\end{example}

\subsection{Lambert $W$ function}\label{SS: Lambert}

 We present in this subsection some notations and elementary facts concerning the Lambert $W$ function, $\operatorname{W}_0$, which will be helpful later on several occasions.

First recall that $\operatorname{W}_0(x)$, for $x>0$ is defined as the unique solution to the equation
\[
\operatorname{W}_0(x)\exp(\operatorname{W}_0(x))=x.
\]
It is well known that
\[
\operatorname{W}_0'(x)=\frac {\operatorname{W}_0(x)}{x(1+\operatorname{W}_0(x))}>0.
\]
By~\cite{HH08} we have the following estimates:
\begin{enumerate}
\item for $x\geq e$
\[
\frac 12\log (x)\leq \operatorname{W}_0(x)\leq \log (x);
\]

\item for $x\geq e$
\[
\log (x)- \log(\log (x))\leq \operatorname{W}_0(x)\leq \log (x)- \frac 12\log(\log (x)).
\]
\end{enumerate}
From the above we can deduce that for $x\geq 0$ it holds:
\begin{equation}\label{estL}
\operatorname{W}_0(x)\leq \max(1,\log x).
\end{equation}
For further information on Lambert $W$ function we refer the reader to~\cite{Mez22}.

\section{Volume Estimation via Capacity}\label{Sec3: DK}

Dinew and Ko\l odziej in~\cite{DK14} proved that the volume of a relatively compact set can be estimated using capacity. They established that for $1 < \alpha < \frac{n}{n-m}$, the following inequality holds:
\[
V_{2n}(K) \leq C(\operatorname{cap}_m(K))^{\alpha}.
\]
This result has proven essential for studying the complex Hessian operator in $\mathbb{C}^n$ and on compact K\"ahler manifolds. However, for our aim in Section~\ref{Sec5: DP}, this result is not sufficient. Consequently, we aim to refine the Dinew-Ko\l odziej estimate in Theorem~\ref{DK}. First, we list two known facts required for the proof.

\begin{lemma}[Theorem 1 in~\cite{Ko96}]\label{kol}
Let $h$ be an increasing function such that
\[
\int_1^{\infty} \frac{1}{y h^{\frac{1}{n}}(y)} \, dy < \infty,
\]
and let $\mu$ be a non-negative measure on a bounded, strictly pseudoconvex domain $\Omega \subset \mathbb{C}^n$. Assume that there exists a constant $A > 0$ such that for any $v \in \mathcal{PSH}(\Omega) \cap \mathcal{C}(\bar{\Omega})$, with $v = 0$ on $\partial \Omega$ and $\int_{\Omega} (dd^c v)^n \leq 1$ it holds:
\[
\int_{\Omega} (-v)^n h(-v) d\mu \leq A .
\]
Then any bounded plurisubharmonic function $u$ with $(dd^c u)^n = \mu$ will satisfy
\[
\|u\|_{\infty} \leq B(A, h),
\]
where the constant $B(A, h)$ does not depend on $\mu$.
\end{lemma}

\begin{lemma}[Theorem 4.1 in \cite{ACKPZ09}]\label{A-C-K-P-Z}
Let $v \in \mathcal{F}_n$ with $\int_{\Omega} (dd^c v)^n \leq 1$. For all $s > 0$, it follows that
\[
V_{2n}(\{v \leq -s\}) \leq C_n (1 + s)^{n-1} \exp(-2ns).
\]
\end{lemma}

We are now ready to present the first result of this paper.

\begin{theorem}\label{DK}
Let $\Omega$ be a bounded $m$-hyperconvex domain in $\mathbb C^n$. Then for any $0<\epsilon\leq \frac {n+1}{3n}$ there exist constants $C_1, C_2>0$ such that for all $K\Subset \Omega$ it  holds:
\begin{equation}\label{est}
V_{2n}(K)\leq C_1\operatorname {cap}_m(K)^{\frac {n}{n-m}}\operatorname{W}_0\left(C_2\operatorname {cap}_m(K)^{\frac {-1}{m(1+\epsilon)}}\right)^{\frac {nm(1+\epsilon)}{n-m}},
\end{equation}
where $\operatorname{W}_0$ is the Lambert $W$ function.
\end{theorem}
\begin{proof}
Without loss of generality, we may assume that $\Omega$ is a bounded strictly pseudoconvex set with $V_{2n}(\Omega)\leq 1$. Otherwise, there exists some $R > 0$ such that $\Omega \subset B(0, R)$ and $\operatorname{cap}_{K(0, R)}(K) \leq \operatorname{cap}_{\Omega}(K)$. Consider a compact set $K$ with $V_{2n}(K) > 0$. If not, there is nothing to prove.

Fix $0 < \epsilon \leq \frac{n+1}{3n}$ and define the function
\[
F(t) = t^{-1} (-\log t)^{-n-n\epsilon}.
\]

Let $\varphi$ be a plurisubharmonic solution to the Monge-Amp\`ere equation:
\[
(dd^c \varphi)^n = F(V_{2n}(K)) \chi_K dV_{2n}, \quad \text{and } \varphi = 0 \text{ on } \partial \Omega.
\]
Then, by the inequality between mixed Monge-Amp\`ere measures (see~\cite{Din09}), we obtain
\begin{equation}\label{mix}
\HH(\varphi) \geq F^{\frac{m}{n}}(V_{2n}(K)) \chi_K dV_{2n}.
\end{equation}
Define
\[
\Phi(t) = \exp\left(2n(1-\epsilon)(t+1)^{\frac{1}{n+n\epsilon}}\right)
\]
and note that, given the choice of $\epsilon$, the function $\Phi$ is increasing and convex.

We apply Lemma~\ref{kol} with $h(t) = (\log(1+t))^{n+n\epsilon}$ and $\mu = F(V_{2n}(K)) \chi_K dV_{2n}$. Consider $v \in \mathcal{PSH}(\Omega) \cap \mathcal{C}(\bar{\Omega})$, with $v = 0$ on $\partial \Omega$ and $\int_{\Omega} (dd^c v)^n \leq 1$. From Theorem~\ref{hol}, (\ref{o2}) and (\ref{o1}), we have
\begin{multline}\label{orlicz}
\int_{\Omega} (-v)^n h(-v) d\mu = \int_{\Omega} (-v)^n h(-v) F(V_{2n}(K)) \chi_K dV_{2n} \\
\leq \|(-v)^n h(-v)\|_{\Phi} F(V_{2n}(K)) \|\chi_K\|_{\Phi^*}^0 \\
\leq \|(-v)^n h(-v)\|_{\Phi}^0 F(V_{2n}(K)) \|\chi_K\|_{\Phi^*}^0 \\
\leq \left(\int_{\Omega} \Phi((-v)^n h(-v)) dV_{2n} + 1\right) F(V_{2n}(K)) V_{2n}(K) \Phi^{-1}\left(\frac{1}{V_{2n}(K)}\right),
\end{multline}
where $\Phi^{-1}$ is the inverse function,
\[
\Phi^{-1}\left(\frac{1}{s}\right) = \left(-\frac{\log s}{2n(1-\epsilon)}\right)^{n+n\epsilon} - 1 \leq \left(-\frac{\log s}{2n(1-\epsilon)}\right)^{n+n\epsilon}.
\]
Then
\begin{equation}\label{orlicz2}
F(V_{2n}(K)) V_{2n}(K) \Phi^{-1}\left(\frac{1}{V_{2n}(K)}\right) = \frac{1}{(2n(1-\epsilon))^{n+n\epsilon}} =\tilde C,
\end{equation}
where $\tilde C$ is a constant independent of $K$.

On the other hand, by Lemma~\ref{A-C-K-P-Z}, we have
\begin{multline}\label{orlicz3}
\int_{\Omega} \Phi((-v)^n h(-v)) dV_{2n} \leq C_n + \sum_{s=0}^{\infty} \int_{\{-s-1 < v < -s\}} \Phi((-v)^n h(-v)) dV_{2n} \\
\leq C_n\sum_{s=0}^{\infty} (1+s)^{n-1} \exp(-2ns) \exp\left(2n(1-\epsilon)((s+1)^n (\log(2+s))^{n+n\epsilon})^{\frac{1}{n+n\epsilon}}\right) \\
\leq A,
\end{multline}
where the constant $A$ does not depend on $v$.

Combining (\ref{orlicz}), (\ref{orlicz2}), and (\ref{orlicz3}), we conclude
\[
\int_{\Omega} (-v)^n h(-v) d\mu \leq \tilde C(A + 1),
\]
thereby, by Lemma~\ref{kol}, a constant $d$, independent of $K$, exists such that the solution $\varphi$ satisfies $\|\varphi\|_{\infty} \leq \frac{1}{d}$.

Now, let $\psi = d\varphi$, then
\begin{multline}\label{final}
\operatorname{cap}_m(K) \geq \int_{\Omega} \HH(\psi) = d^m \int_K \HH(\varphi)
\geq d^m \int_K F^{\frac{m}{n}}(V_{2n}(K)) dV_{2n} = \\
= d^m (V_{2n}(K))^{1 - \frac{m}{n}} (-\log(V_{2n}(K)))^{-(1+\epsilon)m}.
\end{multline}
Define the function $G_{p,q}(t) = t^q (-\log t)^p$, for $p < 0$, $q > 0$, which is increasing for $t \in (0, 1)$, with its inverse given by
\[
G_{p,q}^{-1}(s) = \frac{\left(-\frac{q}{p}\right)^{\frac{p}{q}} s^{\frac{1}{q}}}{\operatorname{W}_0\left(-\frac{q}{p} s^{\frac{1}{p}}\right)^{\frac{p}{q}}}.
\]

In our scenario $p = -(1+\epsilon)m$ and $q = 1-\frac{m}{n}$. Finally, it follows from (\ref{final}) that
\begin{align*}
V_{2n}(K) &\leq G_{p,q}^{-1}\left(d^{-m} \operatorname{cap}_m(K)\right) \\ & = C_1 \operatorname{cap}_m(K)^{\frac{n}{n-m}} \operatorname{W}_0(C_2 \operatorname{cap}_m(K)^{\frac{-1}{(1+\epsilon)m}})^{\frac{nm(1+\epsilon)}{n-m}},
\end{align*}
for some universal constants $C_1$ and $C_2$. The proof is complete.
\end{proof}

\begin{corollary}\label{cor}
Let $\Omega$ be a bounded $m$-hyperconvex domain in $\mathbb C^n$. Then for any $0<\epsilon\leq \frac {n+1}{3n}$ there exist constants $D_1, D_2>0$ such that for all $K\Subset\Omega$ it holds:
\begin{equation}\label{est2}
V_{2n}(K)\leq D_1\operatorname {cap}_m(K)^{\frac {n}{n-m}}\max(1,1-D_2\log (\operatorname {cap}_m(K))^{\frac {nm(1+\epsilon)}{n-m}}
\end{equation}
\end{corollary}
\begin{proof}
It follows from Theorem~\ref{DK} and the estimate (\ref{estL}) for the Lambert $W$ function.
\end{proof}

\section{Continuous Solution}\label{Sec: Continuous Solution}

Let $n \geq 2$ and $1 \leq m \leq n$. In this section, we confine our discussion to a bounded strictly $m$-pseudoconvex domain $\Omega \subset \mathbb{C}^n$, as defined in Definition~\ref{stricm-hyp}. Consider the following Dirichlet problem for the complex Hessian equation given a density function $f$ and a boundary value function $g \in \mathcal{C}(\partial \Omega)$:
\begin{equation}
\label{Sec5: DP}
\begin{aligned}
\HH(U(f,g)) &= f dV_{2n}, \\
\lim_{z \to w} U(f,g)(z) &= g(w), \quad \text{for all } w\in\partial \Omega,
\end{aligned}
\end{equation}
where $dV_{2n}$ represents the Lebesgue measure in $\mathbb{R}^{2n}$. Theorem~\ref{t1} establishes that if the density function $f$ is in $L^{\frac{n}{m}}(\log L)^{\alpha}$ with $\alpha > 2n$, then the solution exists and is continuous on $\bar{\Omega}$. The proof is using the following result recently proved in~\cite{CZ23}.

\begin{theorem}\label{C-Z}
Let $\Omega \subset \mathbb{C}^n$ be a bounded strictly $m$-pseudoconvex domain and let $\mu$ be a positive finite Borel measure on $\Omega$ such that for all compact sets $K \subset \Omega$ the following holds:
\[
\mu(K) \leq A \operatorname{cap}_m(K) F(\operatorname{cap}_m(K)),
\]
where $F:(0,\infty) \to (0,\infty)$ is a continuous increasing function that satisfies
\[
\int_{0^+} \frac{F(t)^{\frac{1}{m}}}{t} \, dt < \infty.
\]
Then, for any positive continuous boundary function $g \in \mathcal{C}(\partial \Omega)$, there exists a unique continuous solution $U(\mu, g)$ of the Dirichlet problem for the complex Hessian equation (\ref{Sec5: DP}).
\end{theorem}

To provide an $L^{\infty}$-estimate of the solution we shall use the following lemma from~\cite{CZ23}.

\begin{lemma}\label{lem}
Let $h: [0, \infty) \to [0, \infty)$ be a decreasing, right-continuous function with $\lim_{s \to \infty} h(s) = 0$. Let $\eta: [0, \infty) \to [0, \infty)$ be a non-decreasing function that satisfies the integrability condition:
\[
\int_{0^+} \frac{\eta(t)}{t} \, dt < \infty.
\]
Assume for any $t \in [0, 1]$ and any $s > 0$, the inequality
\[
t h(s + t) \leq h(s) \eta(h(s))
\]
holds. Then $h(s) = 0$ for all $s \geq S_{\infty}$, where $S_{\infty}$ is defined as
\[
S_{\infty} = s_0 + e \int_0^{eh(s_0)} \frac{\eta(t)}{t} \, dt,
\]
and $s_0$ is determined by the condition $\eta(h(s_0)) \leq \frac{1}{e}$.
\end{lemma}

Using the above theorem and lemma, we proceed with our proof that the solution is continuous under the stated conditions.

\begin{theorem}\label{t1}
Let $\Omega \subset \mathbb{C}^n$ be a bounded strictly $m$-pseudoconvex domain, let $f \in L^{\frac{n}{m}}(\log L)^{\alpha}$ for $\alpha > 2n$, and let $g \in \mathcal{C}(\partial \Omega)$. Then, the unique solution $U(f, g)$ of the Dirichlet problem for the complex Hessian operator~\eqref{Sec5: DP} is continuous on $\bar{\Omega}$. Moreover, the following estimate holds:
\begin{multline}\label{bound}
\|U(f_1,g_1)-U(f_2,g_2)\|_{\infty} \leq \|g_1-g_2\|_{\infty}+ C_1 \|f_1-f_2\|_{\alpha}^{-\frac{1}{\gamma}}\\
+ C_2 e_{m,m}(U(|f_1-f_2|,0))^{\frac{1}{2m}} \exp\left(C_3 \|f_1-f_2\|_{\alpha}^{-\frac{1}{\gamma}}\right),
\end{multline}
where $\gamma = (1+\epsilon)m - \frac{\alpha m}{n}$ with $0 < \epsilon < \min\left(\frac{n+1}{3n},\frac {\alpha}{n}-2\right)$. Here, $\|f\|_{\alpha}$ denotes the norm in $L^{\frac{n}{m}}(\log L)^{\alpha}$, and $C_1, C_2, C_3$ are positive universal constants. Moreover, $e_{m,m}(u)$ is defined as $\int_{\Omega} (-u)^m \HH(u)$.
\end{theorem}

\begin{proof} Define the function
\[
G_{\alpha,\frac{n}{m}}(t) = (1+t)^{\frac{n}{m}} (\log(1+t))^{\alpha},
\]
which is increasing and convex. Its inverse is given by
\[
G_{\alpha,\frac{n}{m}}^{-1}(s) = \frac{\left(\frac{n}{\alpha m}\right)^{\frac{\alpha m}{n}} s^{\frac{m}{n}}}{\operatorname{W}_0\left(\frac{n}{\alpha m} s^{\frac{1}{\alpha}}\right)^{\frac{\alpha m}{n}}} - 1,
\]
and the function $t \left(G_{\frac{n}{m}, \alpha}^{-1}(\frac{1}{t})+1\right)$ is increasing.
Assuming $0 < \epsilon < \min\left(\frac{n+1}{3n},\frac {\alpha}{n}-2\right)$ and choosing a relatively compact subset $K \subset \Omega$, using the Corollary~\ref{cor}, (\ref{o2}) and (\ref{o1}), we have:
\begin{equation}\label{est3}
\begin{aligned}
\mu(K) &= \int_K f dV_{2n} = \int_{\Omega} f \chi_K dV_{2n} \leq \|f\|_{\alpha} \|\chi_K\|_{G_{\alpha,\frac{n}{m}}^*}^0 \\
&\leq \left(\int_{\Omega} G_{\alpha,\frac{n}{m}}(f) dV_{2n} + 1\right) V_{2n}(K) G_{\frac{n}{m}, \alpha}^{-1}\left(\frac{1}{V_{2n}(K)}\right) \\
&\leq D_1 \operatorname{cap}_m(K) \max(1, 1 - D_2 \log (\operatorname{cap}_m(K)))^{\gamma},
\end{aligned}
\end{equation}
where $\gamma = (1+\epsilon)m - \frac{\alpha m}{n}$. Define
\[
F(t) = \max(1, 1 - D_2 \log t)^{\gamma}.
\]
Note that since by our assumption $\alpha>(2+\epsilon)n$, the function $t^{-1}F^{\frac 1m}(t)$ is locally integrable near zero, therefore by Theorem~\ref{C-Z} there exists a unique solution $U(f, g)$ of the Dirichlet problem for the complex Hessian operator (\ref{Sec5: DP}) which is continuous on $\bar{\Omega}$.

Now we shall provide the $L^{\infty}$-estimate of the solution $U(f,g)$. First, note that it follows from the comparison principle that:
\begin{multline}\label{est5}
|U(f_1, g_1) - U(f_2, g_2)| \leq -U(|f_1 - f_2|, -|g_1 - g_2|) \\ \leq -U(|f_1 - f_2|, 0) + \|g_1 - g_2\|_{\infty}.
\end{multline}
Therefore, it is sufficient to prove the result for the density function belonging to the Orlicz space $L^{\frac{n}{m}}(\log L)^{\alpha}$ and with boundary values equal to zero. Set, $u=U(|f_1 - f_2|,0)$.

From \cite{CZ23}, for all $s, t > 0$, it follows
\begin{equation}\label{est4}
t^m \operatorname{cap}_m(\{u < -s - t\}) \leq \int_{\{u < -s\}} \HH(u) \leq \frac{1}{t^m} e_{m,m}(u).
\end{equation}
Define
\[
h(t) = \operatorname{cap}_m(\{u < -t\})^{\frac{1}{m}},
\]
then (\ref{est3}) and (\ref{est4}) yields
\[
t h(t + s) \leq h(s) \eta(h(s))
\]
with
\[
\eta(t) = D_1^{\frac 1m} \left(\max\left(1, 1 - \frac{D_2}{m} \log t\right)\right)^{\frac{\gamma}{m}}.
\]
By Lemma~\ref{lem}, there exists $S_{\infty}$ such that
\[
h(t) = 0 \quad\text{for } t \geq S_{\infty},
\]
provided $\int_{0^+} \frac{\eta(t)}{t} < \infty$. Note that $\frac {\gamma}{m} < -1$, since $\alpha > (2+\epsilon)n$. The result then follows from the fact that if
\[
\operatorname{cap}_m(\{u<-t\})=0 \Rightarrow V_{2n}(\{u<-t\})=0, \ \text{ for } \ t\geq S_{\infty},
\]
then $u$ is bounded by $S_{\infty}$.

Next, we shall estimate the norm of  $u$. We need to examine $S_{\infty}$ and $s_0$. Assume $s_0$ is such that $\eta(h(s_0)) \leq \frac{1}{e}$, then the expression for $\eta(t)$ can be simplified to, see (\ref{est3}),
\[
\eta(t) = d_1^{\frac 1m} \|f\|_{\alpha}^{\frac{1}{m}} \left(d_2 - \frac 1m \log t\right)^{\frac {\gamma}{m}},
\]
where $d_1, d_2 > 0$ are universal constants. The condition $\eta(h(s_0)) = \frac{1}{e}$ leads us to the equation
\[
d_2 - \frac 1m \log h(s_0) = e^{-\frac{m}{\gamma}}d_1^{-\frac 1{\gamma}} \|f\|_{\alpha}^{-\frac{1}{\gamma}}.
\]
This equality implies that
\begin{multline}\label{int}
e \int_0^{eh(s_0)} \frac{\eta(t)}{t} dt = -\frac{ed_1^{\frac 1m}m^2}{\gamma + m} \|f\|_{\alpha}^{\frac{1}{m}} \left(d_2 - \frac 1m \log (eh(s_0))\right)^{\frac{\gamma + m}{m}} \\
= -\frac{ed_1^{\frac 1m}m^2}{\gamma + m} \|f\|_{\alpha}^{\frac{1}{m}} (e^{-\frac{m}{\gamma}}d_1^{-\frac 1{\gamma}} \|f\|_{\alpha}^{-\frac{1}{\gamma}} - m^{-1})^{\frac{\gamma + m}{m}} \\ \leq C_1 \|f\|_{\alpha}^{\frac{1}{m}} \|f\|_{\alpha}^{-\frac{\gamma + m}{\gamma m}} = C_1 \|f\|_{\alpha}^{-\frac{1}{\gamma}},
\end{multline}
where $C_1 > 0$ is a universal constant.

The condition $\eta(h(s_0)) \leq \frac{1}{e}$ is equivalent to
\[
\operatorname{cap}_m(\{u < -s_0\})^{\frac 1m} = h(s_0) \leq \exp\left(md_2 - m e^{-\frac{m}{\gamma}}d_1^{-\frac 1{\gamma}} \|f\|_{\alpha}^{-\frac{1}{\gamma}}\right).
\]
Therefore, applying (\ref{est4}) with $s = t$, we find $s_0$ that satisfies
\[
\left(\frac{4^m}{s_0^{2m}} e_{m,m}(u)\right)^{\frac 1m} = \exp\left(md_2 - m e^{-\frac{m}{\gamma}}d_1^{-\frac 1{\gamma}} \|f\|_{\alpha}^{-\frac{1}{\gamma}}\right),
\]
which can be rewritten as
\begin{equation}\label{s0}
s_0 = C_2 e_{m,m}(u)^{\frac{1}{2m}} \exp\left(C_3 \|f\|_{\alpha}^{-\frac{1}{\gamma}}\right),
\end{equation}
where $C_2, C_3 > 0$ are universal constants.

To finalize the proof, observe that the estimation (\ref{bound}) follows from (\ref{s0}) and (\ref{int}).
\end{proof}

\begin{remark} Theorem~\ref{t1} states that for $\alpha > 2n$, the solution $U(f,g)$ is continuous up to the boundary of a bounded strictly $m$-pseudoconvex domain $\Omega \subset \mathbb{C}^n$. It is unclear whether our assumption on the power $\alpha$ is optimal. Moreover, whether the condition on the domain $\Omega$ can be relaxed to an $m$-hyperconvex domain remains an open question.
\end{remark}

\section{Bounded solution}\label{section_bs}

In the previous section, we proved in Theorem~\ref{t1} that under the assumption $f \in L^{\frac{n}{m}}(\log L)^{\alpha}$ for $\alpha > 2n$, and $g \in \mathcal{C}(\partial \Omega)$, the unique solution $U(f, g)$ of the Dirichlet problem for the complex Hessian operator~\eqref{Sec5: DP} is continuous on $\bar{\Omega}$. This section focuses on the case where $\alpha > n$ and $\Omega \subset \mathbb{C}^n$ is a bounded, strictly $n$-pseudoconvex domain. In Theorem~\ref{t2}, we show that the solution to~\eqref{Sec5: DP} is bounded, and in Example~\ref{Ex BD}, we observe that when $\alpha \leq n$, the solution may be unbounded.

Lu and Guedj proved a result similar to Theorem~\ref{t2} in the setting of compact K\"ahler manifolds (see Theorem~B and the remark just above the acknowledgments in~\cite{GL23}).

\begin{theorem}\label{t2}
Let $\Omega \subset \mathbb{C}^n$ be a bounded strictly $n$-pseudoconvex domain, and let $f \in L^{\frac{n}{m}}(\log L)^{\alpha}$ for $\alpha > n$ and let $g \in \mathcal{C}(\partial \Omega)$. Then, there exists the unique solution $U(f, g)$ of the Dirichlet problem for the complex Hessian operator~\eqref{Sec5: DP} and it is bounded.
\end{theorem}
\begin{proof}
Without loss of generality, we assume that $g \leq 0$. Assume that  $(1+f)^{n/m} (\log (1+f))^{\alpha} \in L^1$ for some $\alpha > n$, and set $h = (1+f)^{n/m}$. Then $h (\log h)^{\alpha} \in L^1$, and by~\cite{Ko96}, there exists a bounded plurisubharmonic (and hence $m$-subharmonic) solution $v$ to
\[
(dd^c v)^n = h dV_{2n}, \quad \text{and} \;  v = g \; \text{on} \; \partial \Omega.
\]
By the mixed Monge-Amp\`{e}re inequality~\cite{Din09}, we have
\[
(dd^c v)^m \wedge \beta^{n-m} \geq h^{m/n} dV_{2n} \geq f dV_{2n}.
\]
Additionally, as shown in~\cite{ACH18}, $U(0,g) \in \mathcal{SH}_m(\Omega)$ is a maximal and continuous $m$-subharmonic function satisfying
\[
\lim_{z \to w} U(0,g)(z) = g(w) \quad \text{for all} \; w \in \partial \Omega.
\]
It follows from~\cite{EG21} that there exists $U(f,g)$ in the Cegrell class $\mathcal{F}_m$ with generalized boundary value $U(0,g)$, and such that $\HH(U(f,g)) = f dV_{2n}$. Applying the comparison principle~\cite{EG21} gives
\[
v \leq U(f,g) \leq U(0,g),
\]
and thus $U(f,g)$ is a bounded $m$-subharmonic function satisfying
\[
\lim_{z \to w} U(f,g)(z) = g(w) \quad \text{for all} \; w \in \partial \Omega.
\]
\end{proof}

\begin{example}\label{Ex BD}
 Let $B(0,1) \subset \mathbb C^n$ be the unit ball, and consider a radially symmetric density function $f$ with $\alpha \leq n$. Set $g = f^{\frac{m}{n}}$. Let $U_m$ be a solution to $\HH(U_m) = f dV_{2n}$, with $\lim_{z \to w} U_m(z) = 0$ for all $w \in \partial B(0,1)$, and $U_n$ be a solution to $(dd^c U_n)^n = g dV_{2n}$, with $\lim_{z \to w} U_n(z) = 0$ for all $w \in \partial B(0,1)$.

 Then by~\cite{NVT18}, we have
 \[
 \begin{aligned}
   - U_m(z) &= \int_{|z|}^{1} t^{1-\frac{2n}{m}} F(t)^{\frac{1}{m}} dt, \quad F(t) = \frac{1}{2^{2n-m-1}(n-1)!} \int_0^t r^{2n-1} f(r) dr; \\
   - U_n(z) &= \int_{|z|}^{1} t^{-1} G(t)^{\frac{1}{n}} dt, \quad G(t) = \frac{1}{2^{n-1}(n-1)!} \int_0^t r^{2n-1} g(r) dr. \\
 \end{aligned}
 \]
 Using H\"older's inequality, we obtain
 \[
 \begin{aligned}
   G(t) &= \frac{1}{2^{n-1}(n-1)!} \int_0^t r^{2n-1} g(r) dr \\
   &\leq \frac{1}{2^{n-1}(n-1)!} \left( \int_0^t r^{2n-1} f(r) dr \right)^{\frac{m}{n}} \left( \int_0^t r^{2n-1} dr \right)^{\frac{n-m}{n}} \\
   &= C(n,m) F(t)^{\frac{m}{n}} t^{2n-2m},
 \end{aligned}
 \]
 where $C(n,m)$ is a constant depending only on $n$ and $m$. Again, using H\"older's inequality, we have
 \begin{equation}\label{ub}
 \begin{aligned}
   -U_n(z) &= \int_{|z|}^{1} t^{-1} G(t)^{\frac{1}{n}} dt \leq C(n,m)^{\frac{1}{n}} \int_{|z|}^{1} F(t)^{\frac{m}{n^2}} t^{1-\frac{2m}{n}} dt \\
   &\leq C(n,m)^{\frac{1}{n}} \left( \int_{|z|}^{1} t^{1-\frac{2n}{m}} F(t)^{\frac{1}{m}} dt \right)^{\frac{m^2}{n^2}} \left( \int_{|z|}^1 t^{1-\frac{2m}{n}} dt \right)^{\frac{n^2-m^2}{n^2}} \\
   &= D(n,m) (-U_m(z))^{\frac{m^2}{n^2}} \left(1 - |z|^{\frac{2n-2m}{n}} \right)^{\frac{n^2-m^2}{n^2}},
 \end{aligned}
 \end{equation}
 where $D(n,m)$ is a constant depending only on $n$ and $m$.

 Therefore, we have shown that if $U_n$ is unbounded, then $U_m$ is also unbounded. By Example~3.3 in~\cite{AC24}, there exists an unbounded plurisubharmonic function $U_n$ for which $(dd^c U_n)^n = g dV_{2n}$ with $g \in L(\log L)^n$. Therefore, by~\eqref{ub}, the $m$-subharmonic solution $U_m$ to $\HH(U_m) = g^{\frac{n}{m}} dV_{2n}$ is also unbounded.
\end{example}

\end{document}